\documentclass[10pt]{article}   
\usepackage{amsmath,amssymb,amsthm}    				
\usepackage{url}                				
\usepackage{graphicx}           				

\theoremstyle{plain}
\newtheorem{thm}{Theorem}[section]

\newtheorem{prop}[thm]{Proposition}
\newtheorem{lem}[thm]{Lemma}

\theoremstyle{definition}

\newtheorem{defin}[thm]{Definition}

\theoremstyle{remark}

\newtheorem*{proof}{Proof}

\begin{document}

\title{Lower Bound for Convex Hull Area and Universal Cover Problems}
\author{Tirasan Khandhawit, Dimitrios Pagonakis, and Sira Sriswasdi}
\date{}
\maketitle

\begin{abstract} In this paper, we provide a lower bound for an area of the convex hull of points and a rectangle in a plane. We then apply this estimate to establish a lower bound for a universal cover problem. We showed that a convex universal cover for a unit length curve has area at least 0.232239. In addition, we show that a convex universal cover for a unit closed curve has area at least 0.0879873. 
\end{abstract}

\section{Introduction}
One of the open classical problems in discrete geometry is the Moser's Worm problem, which originally asked for ``the smallest cover for any unit-length curve''. In the other words, the question asks for a minimal universal cover for any curve of unit length -- also called unit worm. Although it is not clearly stated in the original problem, in this report, we will only concern ourselves with convex covers.

In 1979, Laidacker and Poole \cite{popo} proved that such minimal cover exists. However, finding
this minimal cover turns out to be much more difficult. Instead, there have been attempts trying to estimate the area of this minimal cover. In 1974, Gerriets and Poole \cite{scho} constructed a rhombus-shaped cover with an area of 0.28870, thus establishing the first upper bound for the problem. Later, Norwood and Poole \cite{tree} improved the upper bound to 0.2738, while Wetzel \cite{pi} conjectured the upper bound of 0.2345. 

On the other hand, the lower bound for the problem has not been as extensively studied. In
1973, Wetzel \cite{cam} showed that any cover has an area at least 0.2194, exploiting the fact that such cover must contain both a unit segment and the \textit{broadworm} \cite{mansur}. This lower bound has only recently improved to 0.227498 \cite{maco}, using the following facts:

\begin{itemize}
\item Any convex universal cover must contain a unit segment, an
  equilateral triangle of side length $\frac{1}{2}$ and a square of side length $\frac{1}{3}$.
\item The minimum area of the convex hull of these three objects provide a lower bound for Moser's Worm Problem.
\end{itemize}

In this paper, we generalize these ideas by considering an arbitrary rectangle instead of the square. In Section 2, we provide a lower bound for the convex hull area of a set of points and a rectangle, and then apply the technique to a universal cover problem. As a result, in Section 3, we improve the lower bound for the Moser's problem to 0.232239. In Section 4, we consider a variation of the Moser's problem: we try to estimate an area of a universal cover for any unit closed curve. Only partial results were established by Wetzel in 1973 when he showed that a translational cover for any unit closed curve must have area between 0.155 and 0.159 \cite{sec}. We are able to show that a convex universal cover must have area at least 0.0879873.

\section{Estimating area of the convex hull of points and rectangles}

Let $\mathcal{P}$ be a polygon with vertices $K_1,\, \dots,\, K_n$, we denote by 
$\mu(\mathcal{P})=\mu(K_1,\, \dots,\, K_n)$ the area of the convex hull of $\mathcal{P}$. We also denote by $\mu(\mathcal{P}_1,\,\mathcal{P}_2,\,\dots,\,\mathcal{P}_m)$ the area of the convex hull of $\mathcal{P}_1 \cup \dots \cup \mathcal{P}_m$, where $\mathcal{P}_i$ are sets in a plane. Next, we define 

\begin{defin}
Given segments $AB$ and $DC$, the \emph{height} of $AB$ with respect
to $DC$ is the length of the perpendicular segment from either $A$ or $B$ to the
parallel of $DC$ passing another point. We denote this
height by $h_{DC}(AB)$ (Figure \ref{Example}).
\end{defin} 

\begin{figure}[htbp]
\begin{center}
	\includegraphics[height = 1.5in]{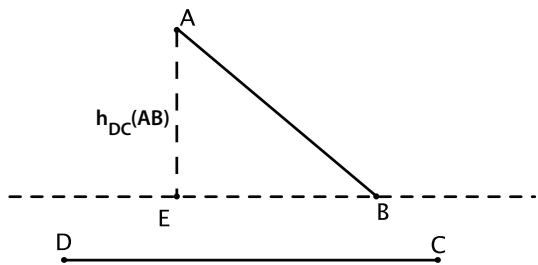}
	\caption{The height of $AB$ with respect to $DC$ is $h_{DC}(AB)$.}
	\label{Example}
\end{center}
\end{figure}

First, we provide a lower bound for the convex hull area of any four points on a plane.

\begin{lem}
Let $E$, $F$, $P$, $Q$ be points in $\mathbb{R}^2$. Then
$\mu(EFPQ)\geq\frac{1}{2}|EF|\,h_{EF}(PQ)$.
\label{R}
\end{lem}

\begin{proof}
Without loss of generality, we can rotate the
figure so that $EF$ is horizontal and $P$ is above $EF$. We can also
relabel points to ensure that $P$ is above $Q$, as well as, $E$ and $F$.
Let $d_1$ be the distance from the point $P$ to $EF$, and $d_2$ the distance from $Q$ to $EF$.
Note that the convex hull of $EFPQ$ always contain triangle $EFP$ and so $\mu(EFPQ)\geq\mu(EFP)$.
If $Q$ also lies above $EF$, then it is clear that $d_1\geq h_{EF}(PQ)$ (Figure \ref{lemmacase2}) and we have $\mu(EFPQ)\geq\mu(EPF) = \frac{1}{2}|EF| d_1\geq \frac{1}{2}|EF|\, h_{EF}(PQ)$.

\begin{figure}[htbp]
\begin{center}
	\includegraphics[height = 2in]{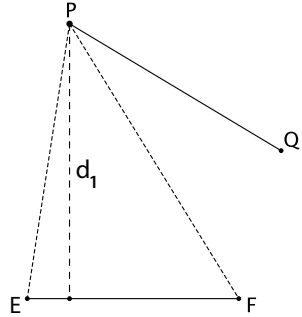}
	\caption{The case where both $P$ and $Q$ lie above $EF$.}
	\label{lemmacase2}
\end{center}
\end{figure}

Otherwise, if $Q$ lies below $EF$, we notice that $EFPQ$ contains both triangles $EPF$ and $EQF$, and the two triangles do not intersect except on $EF$. Hence, $\mu(EFPQ)\geq\mu(EPF)+\mu(EQF)=\frac{1}{2}|EF|\,(d_1+d_2)=\frac{1}{2}|EF|\,h_{EF}(PQ)$, and the equality holds when $EFPQ$ is convex (Figure \ref{lemmacase1}).

\begin{figure}[htbp]
\begin{center}
	\includegraphics[height = 2in]{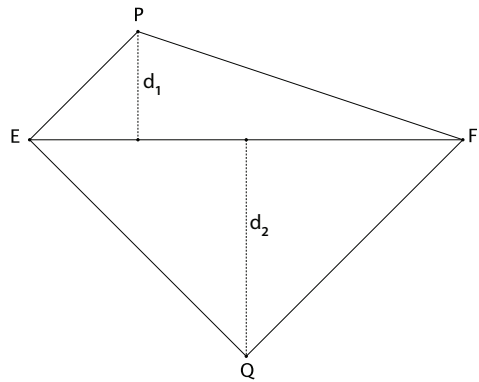}
	\caption{The case where $EF$ lies between $P$ and $Q$.}
	\label{lemmacase1}
\end{center}
\end{figure}

\begin{flushright}
$ \Box $
\end{flushright}

\end{proof}

The next proposition provides a lower bound for the convex hull area of a rectangle and four arbitrary points on the same plane.

\begin{prop}
Let $ABCD$ be a rectangle and $E,\, F,\, P,\, Q$ be points on the same plane. Assume
that $h_{BC}(EF)>|AB|$ and $h_{AB}(PQ)>|BC|$. Then 
\begin{equation}
\mu(ABCDEFPQ)\geq \frac{1}{2}(h_{BC}(EF)-|AB|)\,|BC|+\frac{1}{2}(h_{AB}(PQ)-|BC|)\,|AB|+|AB||BC|
\label{main}
\end{equation}

\end{prop}

\begin{proof}
Without loss of generality, we can assume that the slope of $AB$ is finite and non-negative. Let $\mathcal{V}$ be the strip between extension of $BC$ and $AD$ and $\mathcal{W}$ be the strip between extension of $BC$ and $AD$ (Figure \ref{strips}).

To eliminate redundant cases on the position of four points $E,\, F,\, P,\, Q$ relative to the strips $\mathcal{V}$ and $\mathcal{W}$, we note that by reflecting across the perpendicular bisector of $BC$ and re-labeling $P$ and $Q$ as necessary, we can ensure that $Q$ lies under both $\mathcal{W}$ and $P$. Specifically, if $Q$ initially lies above $\mathcal{W}$, so is $P$ and the reflection will bring both points below $\mathcal{W}$. Otherwise, if $Q$ initially lies inside $\mathcal{W}$, from the assumption that $h_{AB}(PQ)>|BC|$, $P$ must lie above $\mathcal{W}$ and the reflection and re-labeling of $P$ and $Q$ will bring $Q$ below both $\mathcal{W}$ and $P$ as desired.

Similarly, we can ensure that $E$ lies to the left of $\mathcal{V}$ by using reflection across the perpendicular bisector of $AB$ and re-labeling of points $E$ and $F$. Moreover, because the reflection around the perpendicular bisector of $AB$ does not affect the relative positions of $P$, $Q$, and  $\mathcal{W}$, and vice versa, we can obtain both conditions simultaneously. Next, we consider four cases for whether $P$ lies above $\mathcal{W}$ and whether $F$ lies to the right of $\mathcal{V}$.

\textbf{Case 1:} $P$ lies above $\mathcal{W}$ and $F$ lies on the right of $\mathcal{V}$.

\begin{figure}[htbp]
\begin{center}
	\includegraphics[height = 2in]{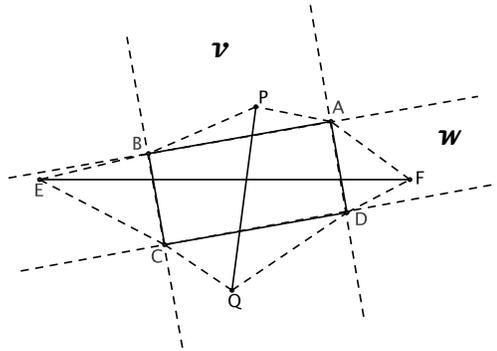}
	\caption{Case with no intersection between triangles.}
	\label{case11}
\end{center}
\end{figure}

If no pair of triangles $BEC$, $CQD$, $DFA$, $APB$ intersects (Figure \ref{case11}), then
we directly obtain the following lower bound for the convex hull area:
\begin{align*}
\mu(ABCDEFPQ) &\geq \mu(BEC) + \mu(DFA) + \mu(APB) + \mu(CQD) + \mu(ABCD) \\
							&= \frac{1}{2}|BC|\,h_1+\frac{1}{2}|AD|\,h_2+\frac{1}{2}|AB|\,g_1+\frac{1}{2}|CD|\,g_2+|AB||BC|,
\end{align*}
where $h_1$, $h_2$, $g_1$, and $g_2$, are the respective heights of triangles $BEC$,
$DFA$, $APB$, and $CQD$. This can also be written as 
\begin{align*}
\mu(ABCDEFPQ) &\geq \frac{1}{2}(h_1+h_2)\,|BC|+\frac{1}{2}(g_1+g_2)\,|AB|+|AB||BC| \\
              &= \text{R.H.S. of inequality \ref{main}}
\end{align*}
and the equality holds when $ABCDEFPQ$ is convex.

When some triangles intersect, because of the conditions on the positions of $E,\, F,\, P$, and $Q$, triangle $BEC$ cannot intersect $DFA$, triangle $APB$ cannot intersect $CQD$, and no three triangles can have nonempty intersection. We pick the case where $DFA\cap CQD\neq \emptyset$ as a representative example and subsequently show that, in general, we can always find a set of disjoint triangles lying inside the convex hull of $ABCDEFPQ$ whose total area is greater than or equal to the sum of areas of triangles $BEC$, $APB$, $DFA$, and $CQD$. 

For the two triangles $DFA$ and $CQD$ to intersect, both $F$ and $Q$ must lie below $\mathcal{W}$ and to the right of $\mathcal{V}$ (Figure \ref{caseinter}). We then draw a line $\mathcal{L}_1$ passing through $F$ parallel to $AD$ and a line $\mathcal{L}_2$ passing through $F$ parallel to $CD$. It is clear that $Q$ must lie either above $\mathcal{L}_2$ or to the right of $\mathcal{L}_1$.

\begin{figure}[htbp]
\begin{center}
	\includegraphics[height = 2in]{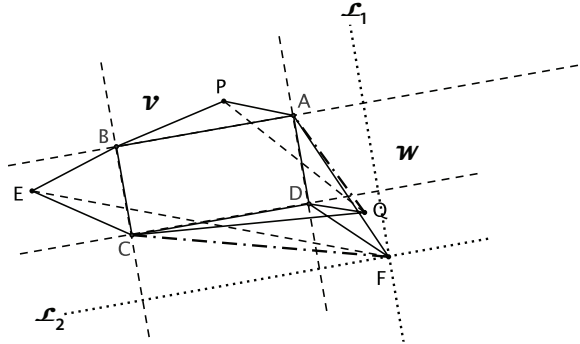}
	\caption{The case where $ADF$ and $CDQ$ intersects.}
	\label{caseinter}
\end{center}
\end{figure}

Hence, we see that either $\mu(CFD) > \mu(CQD)$ or $\mu(DQA) > \mu(DFA)$, and obtain the bound $\max\{\mu(DQA) + \mu(CQD),\ \mu(DFA) + \mu(CFD)\} \geq \mu(CQD) + \mu(DFA)$. Note that none of these triangles can intersect $BEC$, $BPC$, $APB$, or $AEB$ because $P$ lies above $\mathcal{W}$ and $E$ lies to the left of $\mathcal{V}$. If triangles $BEC$ and $APB$ also intersect, we can indenpendently apply the same consideration as above and derive a similar bound.

For the sake of simplicity, consider the case where $BEC$ and $APB$ are disjoint and compute the following lower bound for the convex hull area
\begin{eqnarray*}
\mu(ABCDEFPQ) &\geq &\mu(ABCD)+\mu(APB)+\mu(BEC) \\
							&     &+\max\{\mu(DQA) + \mu(CQD),\ \mu(DFA) + \mu(CFD)\} \\
 							&\geq &\mu(ABCD)+\mu(APB)+\mu(BEC)+\mu(CQD)+\mu(DFA) \\
							&=		&\text{R.H.S. of inequality \ref{main}}
\end{eqnarray*}
and thus inequality \ref{main} holds when $P$ lies above $\mathcal{W}$ and $F$ lies on the right of $\mathcal{V}$.

\textbf{Case 2:} $F$ lies on the right of $\mathcal{V}$ but $P$ does not lie above $\mathcal{W}$.

We now ignore triangle $APB$ and consider only $BEC$, $DFA$, and $CQD$. If $CQD$ intersect $BEC$ or $DFA$, we can use the same argument from Case 1 to show that we can always find a set of disjoint triangles lying inside the convex hull of $ABCDEFPQ$ whose total area is greater than or equal to the sum of areas of triangles $BEC$, $DFA$, and $CQD$. Moreover, since $P$ does not lie above $\mathcal{W}$, the height of $CQD$ with base $CD$ is greater than or equal to $h_{AB}(PQ) - |BC|$, and so we have
\begin{eqnarray*}
\mu(ABCDEFPQ) &\geq &\mu(ABCD)+\mu(BEC)+\mu(DFA)+\mu(CQD) \\
							&\geq &\mu(ABCD)+\mu(BEC)+\mu(DFA) \\
							&     &+\frac{1}{2}(h_{AB}(PQ) - |BC|)\,|AB| \\
							&=		&\text{R.H.S. of inequality \ref{main}}
\end{eqnarray*}

\begin{figure}[htbp]
\begin{center}
	\includegraphics[height = 2in]{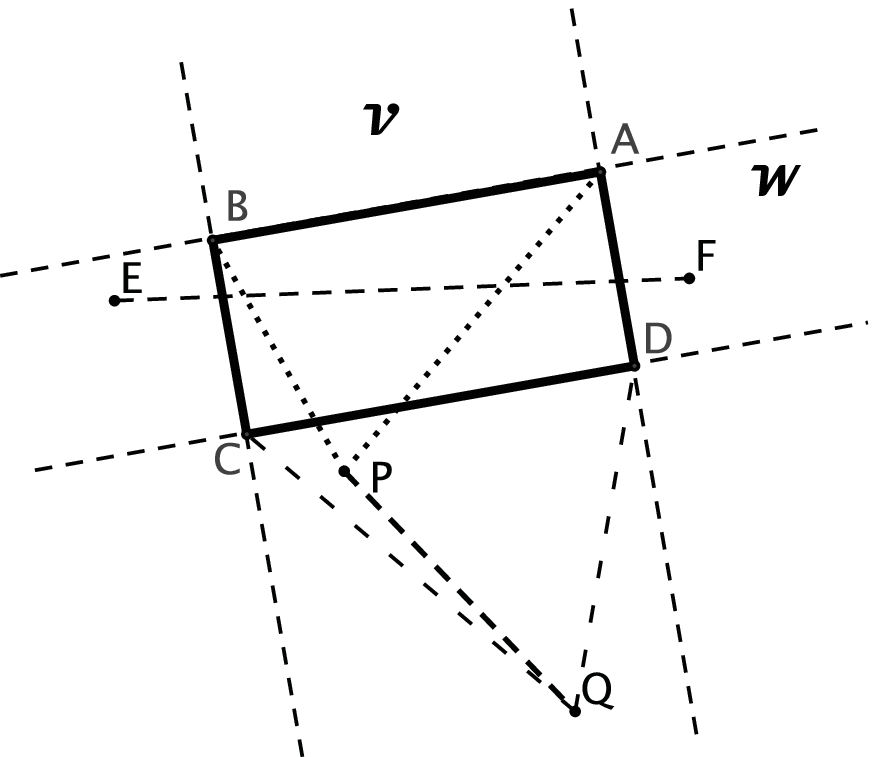}
	\caption{Case 2.}
	\label{case12}
\end{center}
\end{figure}

\textbf{Case 3:} $P$ lies above $\mathcal{W}$ but $F$ does not lie on the right of $\mathcal{V}$.

This case is analogous to Case 2. 

\textbf{Case 4:} $P$ does not above $\mathcal{W}$ and $F$ does not lie on the right of $\mathcal{V}$.

Here we conveniently consider only triangles $BEC$ and $CQD$. By similar argument from Case 1, we can resolve the case where $BEC$ and $CQD$ intersect. Using similar argument from Case 2, we also know that the height of $BEC$ is greater than or equal to $h_{BC}(EF) - |AB|$. Hence,
\begin{eqnarray*}
\mu(ABCDEFPQ) &\geq &\mu(ABCD)+\mu(BEC)+\mu(CQD) \\
							&\geq &\mu(ABCD)+\frac{1}{2}(h_{BC}(EF) - |AB|)\,|BC| \\
							&     &+\frac{1}{2}(h_{AB}(PQ) - |BC|)\,|AB| \\
							&=		&\text{R.H.S. of inequality \ref{main}}
\end{eqnarray*}
\begin{flushright}
$ \Box $
\end{flushright}
\end{proof}
\begin{figure}[htbp]
\begin{center}
	\includegraphics[height = 2in]{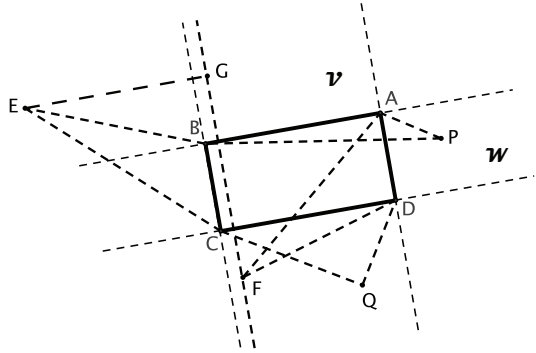}
	\caption{In Case 4 we take only $Q$ and $E$ into consideration.}
	\label{case22}
\end{center}
\end{figure}

Now that we have proved some basic results, we are ready to proceed to the main problem.

\section{Lower bound for a universal cover of curves of unit length}

\subsection{Basic Figures}
Consider a unit segment$\mathcal{L}$ with endpoints $EF$, a V-shaped unit worm $\mathcal{T}$ with vertices $QPR$ and side length $\frac{1}{2}$, and a U-shaped right angle unit worm $\mathcal{R}$ with vertices $ABCD$. To maximize the area of the convex hull of $ABCD$, we let $AB = CD = \frac{1}{2}$ and $BC = AD = \frac{1}{4}$. Another unit worm that we consider is the well known \textit{unit broadworm} \cite{broad}, denoted by $\mathcal{B}$, which was introduced by Schaer in 1968 \cite{mansur} as the broadest unit worm whose minimum width in any direction is given by $b_0$, approximately 0.4389.

We start by introducing parameters to define the positioning of $\mathcal{L}$, $\mathcal{T}$ and $\mathcal{R}$. By rotation, we can assume that $\mathcal{L}$ is horizontal. 
 
Let $O_1 = (x_1,\,y_1)$ be the centroid of the rectangle $\mathcal{R}$. We can always pick the vertex for $A$ so that $A=(x_1,\,y_1)+\frac{\sqrt{5}}{8}(\cos{\alpha},\,\sin{\alpha})$ and $\theta_0 \leq \alpha < \theta_0 + \pi$, where $\theta_0 = \arctan{\frac{1}{2}}$, and then label the rest of the vertices $B$, $C$, and $D$ going counterclockwise. Furthermore, for each configuration of $\mathcal{R}$, the value of $\alpha$ is uniquely defined and we denote it by $\alpha(\mathcal{R})$. Regarding vector as complex number, we see that
\begin{eqnarray*}
	arg\left(\overrightarrow{C A}\right) &=& \alpha \\
	arg\left(\overrightarrow{C D}\right) &=& \alpha - \theta_0.
\end{eqnarray*}

\begin{figure}[htbp]
\begin{center}
	\includegraphics[height = 2in]{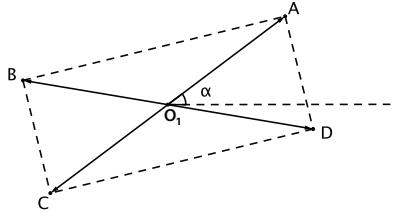}
	\caption{We construct the rectangle by rotating the vector.}
	\label{coord1}
\end{center}
\end{figure}

Let $O_2 = (x_2,y_2)$ be the centroid of the triangle $\mathcal{T}$, we can similarly pick the vertex for $P$ so that $P=(x_2,y_2)+\frac{\sqrt{3}}{6}(\cos{\beta},\sin{\beta})$ and $\frac{\pi}{6} \leq \beta < \frac{5\pi}{6}$, and then label the rest of the vertices $Q$ and $R$ going counterclockwise. Furthermore, for each configuration of $\mathcal{T}$, the value of $\beta$ is uniquely defined and we denote it by $\beta(\mathcal{T})$. Again, regarding vector as complex number, we have
\begin{eqnarray*}
	arg\left(\overrightarrow{QP}\right) &=& \beta - \frac{\pi}{6} \\
	arg\left(\overrightarrow{RP}\right) &=& \beta + \frac{\pi}{6}.
\end{eqnarray*}

Let $\sigma$ be the point reflection across the origin and $\tau$ be the reflection across the $y$-axis. Both transformations keep the segment $\mathcal{L}$ horizontal. We notice that
\begin{eqnarray*}
\alpha\left(\sigma(\mathcal{R})\right) &=& \alpha(\mathcal{R}) \\
\beta\left(\sigma(\mathcal{T})\right) &=& \begin{cases}  \beta(\mathcal{T}) + \frac{\pi}{3} \ \ \mbox{when} \ \ \frac{\pi}{6} \leq \beta(\mathcal{T}) < \frac{\pi}{2} \\
\beta(\mathcal{T}) - \frac{\pi}{3} \ \ \mbox{when} \ \ \frac{\pi}{2} \leq \beta(\mathcal{T}) < \frac{5\pi}{6} \end{cases} \\
\alpha\left(\tau(\mathcal{R})\right) &=& \pi - \alpha(\mathcal{R}) \\ 
\beta\left(\tau(\mathcal{T})\right) &=& \pi - \beta(\mathcal{T}).
\end{eqnarray*}

Hence, with a suitable combination of $\sigma$ and $\tau$, we can ensure that $\theta_0 \leq \alpha(\mathcal{R}) \leq \theta_0 + \frac{\pi}{2}$ and $\frac{\pi}{3} \leq \beta(\mathcal{T}) \leq \frac{2\pi}{3}$.   

\begin{figure}[htbp]
\begin{center}
	\includegraphics[height = 2in]{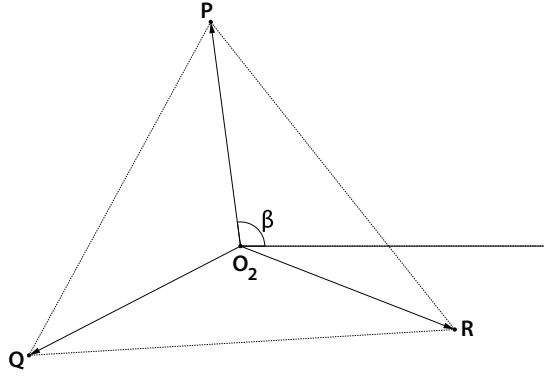}
	\caption{By rotating the vector $\Lambda A$ we can construct the whole triangle.}
	\label{coord2}
\end{center}
\end{figure}

\subsection{Area Estimation}
Here we use the inequalities from the Section 2 to bound the areas of the convex hull of $\mathcal{L}$, $\mathcal{B}$, $\mathcal{T}$ and $\mathcal{R}$. For simplicity, we shall refer to $\alpha(\mathcal{R})$ and $\beta(\mathcal{T})$ by simply $\alpha$ and $\beta$, respectively. The following lemmas provide basic lower bounds on the area of the convex hull of configurations involving line segment $\mathcal{L}$.

\begin{lem}
 $ \mu(\mathcal{L},\, \mathcal{R})\geq \frac{\sqrt{5}}{8} \sin{\alpha} $. \label{lemmaLR}
\end{lem}

\begin{proof}
From Lemma \ref{R} and because $EF$ is horizontal and $arg\left(\overrightarrow{C A}\right) = \alpha$, it follows that $\mu(\mathcal{L},\,\mathcal{R}) = \mu(ABCDEF) \geq \mu(ACEF) \geq \frac{1}{2}|EF|\, h_{EF}(AC) = \frac{\sqrt{5}}{8} \sin{\alpha}$.
\begin{flushright}
$ \Box $
\end{flushright}
\end{proof}

\begin{lem}
$  \mu(\mathcal{L},\,\mathcal{T})\geq \frac{1}{4} \max\left\{ \sin{(\beta - \frac{\pi}{6})},\,\sin{(\beta + \frac{\pi}{6})}\right\}. $
\end{lem}

\begin{proof}
Note that $\mu(\mathcal{L}, \mathcal{T}) \geq \max\left\{ \mu(EFPQ),\, \mu(EFPR) \right\}$, and we also know that $\mu(EFPQ) \geq \frac{1}{4} \sin{(\beta - \frac{\pi}{6})}$ and $\mu(EFPR) \geq \frac{1}{4} \sin{(\beta + \frac{\pi}{6})}$.   
\begin{flushright}
$ \Box $
\end{flushright}
\end{proof} 

Next, we provide a lower bound for the area of the convex hull of configurations involving $\mathcal{L}$, $\mathcal{T}$ and $\mathcal{R}$ together.

\begin{prop}
$\mu(\mathcal{L},\,\mathcal{T},\,\mathcal{R}) \geq \frac{1}{8}\left(\sin(\alpha - \theta_0 + \frac{\pi}{2})+ \sin(\beta - \alpha + \theta_0 + \frac{\pi}{6})\right). $
\end{prop}

\begin{proof}
First notice that $\mu(\mathcal{L},\,\mathcal{T},\,\mathcal{R}) \geq \mu(ABCDEFPR)$, then we apply Proposition \ref{main} to obtain
$$\mu(ABCDEFPR) \geq \frac{1}{2}\left(h_{BC}(EF)\, |BC| + h_{AB}(PR)\, |AB|\right).$$

Since $arg\left(\overrightarrow{B C}\right) = \alpha - \theta_0 + \frac{\pi}{2}$, we have $h_{BC}(EF) = \sin{(\alpha - \theta_0 + \frac{\pi}{2})}$. Similarly, we obtain 
$h_{AB}(PR) = \frac{1}{2} \sin{(\beta - \alpha + \theta_0 + \frac{\pi}{6})}$, and therefore get the desired result.
\begin{flushright}
$ \Box $
\end{flushright} 
\end{proof}

Lastly, we replace $\mathcal{T}$ with the broadworm $\mathcal{B}$ and denote its breadth by $b_0$.

\begin{prop}
$\mu(\mathcal{L},\,\mathcal{R},\,\mathcal{B}) \geq \frac{1}{4}\left(\frac{1}{2}\sin(\alpha - \theta_0 + \frac{\pi}{2})+b_0\right) $. \label{propLRB}
\end{prop}

\begin{proof}
By the definition of breath, there exist points $S$ and $T$ on $\mathcal{B}$ such that $h_{AB}(ST)\geq b_0$. Again, we apply Proposition \ref{main} to obtain
$$\mu(\mathcal{L},\,\mathcal{T},\,\mathcal{B}) \geq \mu(ABCDEFST) \geq \frac{1}{8}\sin(\alpha - \theta_0 + \frac{\pi}{2}) + \frac{b_0}{4}$$
and complete the proof.
\begin{flushright}
$ \Box $
\end{flushright}
\end{proof}

Now we can combine all the lower bounds together to try to minimize $\mu(\mathcal{L},\,\mathcal{R},\,\mathcal{T},\,\mathcal{B})$. We define the following functions:
\begin{itemize}
\item {$p(\alpha)=\frac{\sqrt{5}}{8} \sin{\alpha} $}
\item {$q(\beta)= \frac{1}{4} \max\left\{ \sin{(\beta - \frac{\pi}{6})},\, \sin{(\beta + \frac{\pi}{6})}\right\} $}
\item {$f(\alpha,\beta)= \frac{1}{8}\left(\sin(\alpha - \theta_0 + \frac{\pi}{2})+ \sin(\beta - \alpha + \theta_0 + \frac{\pi}{6})\right)$}
\item {$g(\alpha)= \frac{1}{4}\left(\frac{1}{2}\sin(\alpha - \theta_0 + \frac{\pi}{2})+b_0\right)$}
\item {$F(\alpha,\, \beta) = \max\left\{p(\alpha),\, q(\beta),\,f(\alpha,\beta),\, g(\alpha)\right\}$}
\end{itemize}
It follows immediately that $\mu(\mathcal{L},\,\mathcal{R},\, \mathcal{T},\,\mathcal{B}) \geq F(\alpha,\, \beta)$, and we will find a lower bound for $F(\alpha,\, \beta)$ on the domain $\theta_0 \leq \alpha \leq \theta_0 + \frac{\pi}{2}$ and $\frac{\pi}{3} \leq \beta \leq \frac{2\pi}{3}$.

\begin{prop} $F(\alpha,\,\beta)\geq 0.232239$ on the domain $[\theta_0, \theta_0 + \frac{\pi}{2}] \times [\frac{\pi}{3} , \frac{2\pi}{3}]$. \label{result}
\end{prop}

\begin{proof} We consider the possible values of $F(\alpha, \beta)$ in four cases.

\noindent\textbf{Case 1:} $0.980693572 < \alpha \leq \theta_0 + \frac{\pi}{2}$ 

Clearly we have $\frac{\sqrt{5}}{8}\sin{\alpha} > 0.23223900008$.
\newline\newline
\noindent\textbf{Case 2:} $\theta_0 \leq \alpha < 0.663720973$

We have $\frac{1}{4}\left(\frac{1}{2}\sin(\alpha - \theta_0 + \frac{\pi}{2})+b_0\right) > 0.232239000003$.
\newline\newline
\noindent\textbf{Case 3:} $\frac{\pi}{3} \leq \beta < \frac{\pi}{2} - 0.1443850667\,$ or $\,\frac{\pi}{2} + 0.1443850667 < \beta \leq \frac{2\pi}{3}$

We have $\frac{1}{4}\sin(\beta- \frac{\pi}{6}) > 0.232239000012$ when $\frac{\pi}{3} \leq \beta < \frac{\pi}{2} - 0.1443850667$, and $\frac{1}{4}\sin(\beta+\frac{\pi}{6}) > 0.232239000012$ when $\frac{\pi}{2} + 0.1443850667 < \beta \leq \frac{2\pi}{3}$.
\newline\newline
\noindent\textbf{Case 4:} $0.663720972 \leq \alpha \leq 0.980693573\,$ and $\,\frac{\pi}{2} - 0.1443850668 \leq \beta \leq \frac{\pi}{2} + 0.1443850668$

Consider $f(\alpha,\,\beta)$ on this domain, which is a product of closed and bounded intervals. We can check that there is no local minimum except possibly at the corners and compute 
\begin{align*}
	f(0.663720972,\,\pi/2 - 0.1443850668) &= 0.245506 \\
	f(0.663720972,\,\pi/2 + 0.1443850668) &= 0.234071 \\
	f(0.980693573,\,\pi/2 - 0.1443850668) &= 0.232475 \\
	f(0.980693573,\,\pi/2 + 0.1443850668) &= 0.232239210
\end{align*}
Thus, $f(\alpha,\,\beta) \geq 0.232239210$ on this domain

Therefore, $F(\alpha,\,\beta)= \max\left\{p(\alpha), q(\beta),f(\alpha,\,\beta),g(\alpha)\right\} \geq 0.232239$
\begin{flushright}
$ \Box $
\end{flushright}

\end{proof}

Hence, we have established a new lower bound for the Moser's problem.

\section{Lower bound for a universal cover of unit closed curves}

We now consider a universal cover for any unit closed curve. Denote the segment of length $\frac{1}{2}$, the circle with unit circumference, and a square of side length $\frac{1}{4}$ by $\mathcal{L}$, $\mathcal{C}$, and $\mathcal{R}$, respectively. We parameterize the orientation of $\mathcal{L}$ and $\mathcal{R}$ in essentially the same way as before.

The circle $\mathcal{C}$ imitates the role of the broadworm as the $\mathcal{C}$ has width $\frac{1}{\pi}$ in every direction. Additionally, we can assume that $\frac{\pi}{4} \leq \alpha \leq \frac{\pi}{2}$ because of the symmetry of the square. Then we have the following
\begin{prop}
$i)$ $\mu(\mathcal{L},\mathcal{R}) \geq \frac{\sqrt{2}}{16} \sin(\alpha)\,$ and $\,ii)$ $\mu(\mathcal{L},\mathcal{R},\mathcal{C}) \geq \frac{1}{8}\left(\frac{1}{2}\sin(\alpha + \frac{\pi}{4}) + \frac{1}{\pi}\right)$
\end{prop}

\begin{proof} The prove is exactly the same as in Lemma \ref{lemmaLR} and Proporsition \ref{propLRB}.
\end{proof}

We now define
$$G(\alpha) = \max \left\{ \frac{\sqrt{2}}{16} \sin(\alpha),\, \frac{1}{8}\left(\frac{1}{2}\sin(\alpha + \frac{\pi}{4}) + \frac{1}{\pi}\right) \right\}$$

Then we have,

\begin{prop}
$G(\alpha) \geq 0.0879873$ on the domain $[\frac{\pi}{4},\,\frac{\pi}{2}]$.
\end{prop}

\begin{proof} We divide the domain into two parts

\noindent\textbf{Case 1:} When $1.4755040221 < \alpha \leq \frac{\pi}{2}$, we see that $\frac{\sqrt{2}}{16} \sin(\alpha) > 0.08798734$

\noindent\textbf{Case 2:} When $\frac{\pi}{4} \leq \alpha < 1.4755040222$, we have $\frac{1}{8}\left(\frac{1}{2}\sin(\alpha + \frac{\pi}{4}) + \frac{1}{\pi}\right) >0.08798739$ 
\end{proof}

Hence we establish a lower bound for a universal cover of unit closed curves. 

\noindent\textit{Remark}: In this case, adding an equilateral triangle of side $\frac{1}{3}$ does not improve the lower bound.

\textsc{Tirasan Khandhawit} 

\textsc{Department of Mathematics, M.I.T., Cambridge, MA 02139}

E-mail address: \texttt{tirasan@math.mit.edu}

\textsc{Dimitrios Pagonakis} 

\textsc{22 Finikos Str. Iraklion Crete, Greece 71305}

E-mail address: \texttt{jimicemang@hotmail.com}

\textsc{Sira Sriswasdi} 

\textsc{Graduate Group in Genomics and Computational Biology,}

\textsc{University of Pennsylvania
Philadelphia, PA 19104}

E-mail address: \texttt{sirasris@mail.med.upenn.edu}

\end{document}